\pdfoutput=1
\RequirePackage{ifpdf}
\ifpdf 
\documentclass[pdftex]{sigma}
\else
\documentclass{sigma}
\fi

\usepackage{tikz}

\numberwithin{equation}{section}

\newtheorem{Theorem}{Theorem}[section]
\newtheorem{Proposition}[Theorem]{Proposition}
\newtheorem{Conjecture}[Theorem]{Conjecture}
\newtheorem{Question}[Theorem]{Question}
 { \theoremstyle{definition}
\newtheorem{Definition}[Theorem]{Definition}
\newtheorem{Example}[Theorem]{Example}}

\begin{document}

\allowdisplaybreaks

\newcommand{\arXivNumber}{1909.13002}

\renewcommand{\PaperNumber}{044}

\FirstPageHeading

\ShortArticleName{Higher Rank $\hat{Z}$ and $F_K$}

\ArticleName{Higher Rank $\boldsymbol{\hat{Z}}$ and $\boldsymbol{F_K}$}

\Author{Sunghyuk PARK}

\AuthorNameForHeading{S.~Park}

\Address{California Institute of Technology, Pasadena, USA}
\Email{\href{mailto:spark3@caltech.edu}{spark3@caltech.edu}}

\ArticleDates{Received January 15, 2020, in final form May 11, 2020; Published online May 24, 2020}

\Abstract{We study $q$-series-valued invariants of $3$-manifolds that depend on the choice of a~root system $G$. This is a natural generalization of the earlier works by Gukov--Pei--Putrov--Vafa [arXiv:1701.06567] and Gukov--Manolescu [arXiv:1904.06057] where they focused on $G={\rm SU}(2)$ case. Although a full mathematical definition for these ``invariants'' is lacking yet, we define $\hat{Z}^G$ for negative definite plumbed $3$-manifolds and $F_K^G$ for torus knot complements. As in the $G={\rm SU}(2)$ case by Gukov and Manolescu, there is a surgery formula relating~$F_K^G$ to~$\hat{Z}^G$ of a Dehn surgery on the knot~$K$. Furthermore, specializing to symmetric representations, $F_K^G$ satisfies a recurrence relation given by the quantum $A$-polynomial for symmetric representations, which hints that there might be HOMFLY-PT analogues of these $3$-manifold invariants.}


\Classification{57K16; 57K31; 81R50}

\section{Introduction}
Categorification of the Chern--Simons theory is one of the most exciting open questions in quantum topology. While homology theories categorifying quantum link invariants are fairly well-understood by now, whether there is a homology theory categorifying the Witten--Reshetikhin--Tureav (WRT) invariants is still widely open. One approach to this problem comes from physics. Using physical intuition, Gukov, Putrov and Vafa \cite{GPV} and Gukov, Pei, Putrov and Vafa \cite{GPPV} conjectured that WRT invariants can be decomposed into categorifiable ``homological blocks'' often denoted by $\hat{Z}$. These are integer coefficient $q$-series and are supposed to be the graded Euler characteristic of a conjectural homological invariant $\mathcal{H}^{j,k}_{b,{\rm BPS}}$ so that
\begin{gather*}
\hat{Z}_b(Y;q) = \sum_{j,k}(-1)^j q^k \operatorname{rank}\mathcal{H}^{j,k}_{b,{\rm BPS}}(Y).
\end{gather*}
Furthermore, in~\cite{GPPV}, they gave a definition of $\hat{Z}$ for negative definite plumbed $3$-manifolds.

More recently, Gukov and Manolescu \cite{GM} studied an analog of $\hat{Z}$ for knot complements, denoted by $F_K(x,q) := \hat{Z}\big(S^3 \setminus K\big)$, where $x$ parametrizes the boundary condition, namely the holonomy eigenvalue along the meridian in the complex Chern--Simons theory. Their motivation was to study $\hat{Z}$ of $3$-manifolds described as Dehn surgery on knots, and in fact they demonstrated that $\hat{Z}$ behaves well under gluing pieces of $3$-manifolds along their toral boundaries.

As these exciting developments were focused on ${\rm SU}(2)$ as the gauge group, a natural question is whether they can be generalized to other gauge groups. One strong motivation for this question comes from the fact that, quantum ${\rm SU}(N)$ link invariants (and their categorifications) exhibit a nice regularity under the change of rank, which often can be lifted to an independent variable $a=q^N$ (the third grading), as in HOMFLY-PT polynomials (and HOMFLY-PT homology, as conjectured in~\cite{DGR}). As we show in this paper, $\hat{Z}$ and $F_K$ have higher rank analogues, and moreover there is a certain regularity under the change of rank, which hints that there might be HOMFLY-PT analogues of these $3$-manifold invariants.

\subsection*{Summary of the paper}
The purpose of this paper is to extend these previous works~\cite{GM, GPPV} to arbitrary root system $G$, thereby studying the higher rank analogues of $\hat{Z}$ and $F_K$.

We start by studying the higher rank analogue of $\hat{Z}$ in Section~\ref{sec:higherrankZhat}. Our main result is Definition~\ref{def:Zhatmaindef}, where we give a definition of $\hat{Z}_b^G$ for (weakly) negative definite plumbed $3$-manifolds. We show, in Theorem~\ref{thm:Neumanninvar}, that $\hat{Z}_b^G$ is indeed an invariant by proving that it is invariant under Neumann moves. We also provide some examples of Seifert manifolds and express their $\hat{Z}_b^G$ in terms of higher rank false theta functions.

We turn our attention to the higher rank analogue of $F_K$ in Section~\ref{sec:higherrankFK}. Since torus knot complements can be expressed as plumbings, we can deduce~$F_K^G$ for torus knots from Definition~\ref{def:Zhatmaindef}. An explicit expression is given in Theorem~\ref{thm:FKfortorusknots}. A conjectural higher rank surgery formula is also presented.

Then in Section \ref{sec:largeN}, we specialize our higher rank $F_K$ to symmetric representations. For simplicity we set $G={\rm SU}(N)$. Upon specialization to symmetric representations, we experimentally check that $F_K^{{\rm SU}(N),{\rm sym}}$ is annihilated by the quantum $A$-polynomial for symmetric representations, the $a$-deformed quantum $A$-polynomial with specialization $a=q^N$. From this observation, we are naturally led to speculate the existence of the HOMFLY-PT analogue of $F_K$. Some of the expected properties of the HOMFLY-PT analogue of $F_K$ are stated in Conjecture~\ref{HOMFLYPTFK}.

\subsection*{Notations and conventions}
We follow the convention used in \cite{GM} for knots, 3-manifolds, and colored Jones polynomials. Throughout this article, $G$ is a connected, simply connected\footnote{While we believe that our work can be further generalized to non-simply connected gauge groups by considering the lattice of characters of the maximal torus instead of the weight lattice $P$, in this article we focus on simply connected case for simplicity.} semisimple Lie group with the root system $\Delta\subset \mathfrak{h}^*$, $Q\subset \mathfrak{h}^*$ is the root lattice, $Q^\vee\subset \mathfrak{h}$ is the coroot lattice, $P\subset \mathfrak{h}^*$ is the weight lattice, $W$ is the Weyl group, $\Delta^+$ is the set of positive roots, $\rho$ denotes the Weyl vector (half-sum of positive roots), and the letters~$\alpha$ and~$\omega$ will be reserved for roots and fundamental weights. The inner product $(\cdot,\cdot)$ on $\mathfrak{h}^*$ is the standard one normalized such that $(\alpha,\alpha)=2$ for short roots~$\alpha$. The length of a Weyl group element $w\in W$ will be denoted by~$l(w)$. We use the letter $\rm B$ for the linking matrix of a plumbed 3-manifold, and $\sigma=\sigma({\rm B})$ and $\pi=\pi({\rm B})$ denote the signature and the number of positive eigenvalues of ${\rm B}$, respectively. For a multi-index monomial, we use the following notation
\[x^\beta := \prod_{1\leq j\leq r}x_j^{(\beta,\omega_j)},\]
where $r = \operatorname{rank}G$ and $\beta\in P$. When it comes to $q$-series, often we do not bother to fix the overall power of~$q$, and just use the notation $\cong$ for equivalence up to sign and overall power of~$q$.

\section[Higher rank $\hat{Z}$]{Higher rank $\boldsymbol{\hat{Z}}$}\label{sec:higherrankZhat}
\subsection[The set of labels $\mathcal{B}$]{The set of labels $\boldsymbol{\mathcal{B}}$}\label{spint}

Before getting into the definition of $\hat{Z}^G_b$ for negative definite plumbings, we have to first describe what the labels~$b$ are. In case of $G={\rm SU}(2)$, these labels $b$ were ${\rm Spin}^c$-structures (at least for plumbings on trees), as clarified by~\cite{GM}. That is, conjecturally $\hat{Z}^{{\rm SU}(2)}$ is an invariant for 3-manifolds decorated by ${\rm Spin}^c$-structures. From this, it is natural to expect that, as we generalize~$G$ to an arbitrary root system, $\hat{Z}^G$ will be an invariant of 3-manifolds decorated by structures analogous to ${\rm Spin}^c$-structures.
\begin{Definition}
 For a plumbed 3-manifold $Y=Y(\Gamma)$, define
 \begin{gather*}
 \mathcal{B}^G(Y) := \big(Q^V+\delta\big)/{\rm B}Q^V,
 \end{gather*}
 where $V = V(\Gamma)$ is the set of vertices, and $\delta_v = (2-\deg v)\rho$.
\end{Definition}
This is essentially a generalization of ${\rm Spin}^c$-structures, in a sense that $\mathcal{B}^{{\rm SU}(2)}(Y)\cong{\rm Spin}^c(Y)$ canonically. Recall that ${\rm Spin}^c(Y)$ is affinely isomorphic to $H^2(Y)$ and admits a $\mathbb{Z}_2$ action by conjugation. Similarly, two of the main features of $\mathcal{B}^G(Y)$ are that it is affinely isomorphic to~$H^2(Y;Q)$ and that it admits an action by the Weyl group $W$ (and hence carries an action by~$H^2(Y;Q)\rtimes W$).

\subsection[Higher rank $\hat{Z}$ for negative definite plumbings]{Higher rank $\boldsymbol{\hat{Z}}$ for negative definite plumbings}\label{subsec:zhatdef}

Plumbed $3$-manifolds are $3$-manifolds naturally associated to each graph whose vertices are decorated by integers. Assume for simplicity that the plumbing graph $\Gamma$ is a tree. Given a~decorated tree $\Gamma$, one can make a framed link $L_\Gamma$ by replacing each vertex by an unknot whose framing is the decoration of the vertex, and by linking any two of the unknots (in the simplest possible way) whenever the corresponding vertices are connected by an edge. The plumbed $3$-manifold with plumbing graph $\Gamma$ is the $3$-manifold obtained as the Dehn surgery on~$L_\Gamma$. Following~\cite{GM}, we call the linking matrix ${\rm B}$ of $L_\Gamma$ \emph{weakly negative definite} if~${\rm B}^{-1}$ is negative definite on the subspace spanned by vertices of degree~$\geq 3$.

We present here a formula for $\hat{Z}$ for (weakly) negative definite plumbed manifolds, with arbitrary root system $G$. The motivation for this definition is a heuristic decomposition of WRT invariants which relies heavily on the Gauss sum reciprocity formula; see Appendix \ref{Gauss}.
\begin{Definition}[higher rank $\hat{Z}$ for negative-definite plumbings]\label{def:Zhatmaindef}
For a plumbed 3-manifold $Y$ with a weakly negative definite linking matrix ${\rm B}$\footnote{We follow the convention of \cite{GM} where every edge is positively oriented (i.e., the linking is $+1$). Changing orientation of the edges will change the overall sign. The prefactor $(-1)^{|\Delta^+|\pi}$ takes care of the change of orientation of the edges under Neumann moves.} and a choice of $b\in \mathcal{B}^G(Y)$, define
\begin{gather}
\hat{Z}_b^G(Y;q) := (-1)^{|\Delta^+|\pi}q^{\frac{3\sigma-\operatorname{Tr}{\rm B}}{2}(\rho,\rho)} \nonumber \\
\hphantom{\hat{Z}_b^G(Y;q) :=}{} \times{\rm v.p.}\int_{|x_{vj}|=1}\prod_{v\in V}\prod_{1\leq j\leq r}\frac{{\rm d}x_{vj}}{2\pi{\rm i} x_{vj}}\left(\sum_{w\in W}(-1)^{l(w)} x_v^{w(\rho)}\right)^{2-\deg v}\!\! \Theta_b^{-{\rm B}}\big(x^{-1},q\big), \!\!\!\label{integralZhat}
\\
\Theta_b^{-{\rm B}}\big(x^{-1},q\big) := \sum_{\ell\in {\rm B} Q^{V}+b}q^{-\frac{1}{2}(\ell,{\rm B}^{-1}\ell)}\prod_{v\in V}x_v^{-\ell_v}. \nonumber 
\end{gather}
In particular, in case $G = {\rm SU}(N)$, this takes the following simple form:
\begin{gather*}
\hat{Z}_b^{{\rm SU}(N)}(Y;q) = (-1)^{\frac{N(N-1)}{2}\pi}q^{\frac{3\sigma -\operatorname{Tr}{\rm B}}{2}\frac{N^3-N}{12}}\\
\hphantom{\hat{Z}_b^{{\rm SU}(N)}(Y;q) =}{} \times {\rm v.p.}\oint_{|x_{vj}|=1}\prod_{v\in V}\prod_{1\leq j\leq N-1}\frac{{\rm d}x_{vj}}{2\pi{\rm i} x_{vj}} F_{3d}(x)\Theta_{2d}^{b}(x,q)
\end{gather*}
with
\begin{gather*}
F_{3d}(x) := \prod_{v\in V}\left(\sum_{w\in W}(-1)^{l(w)} \prod_{1\leq j\leq N-1}x_{vj}^{(\omega_j,w(\rho))}\right)^{2-\deg v}\nonumber\\
\hphantom{F_{3d}(x)}{} = \prod_{v\in V}\left( \prod_{1\leq j < k\leq N}\big(y_{vj}^{1/2}y_{vk}^{-1/2}-y_{vj}^{-1/2}y_{vk}^{1/2}\big) \right)^{2-\deg v},\\
\Theta_{2d}^{b}(x,q) :=\sum_{\ell\in {\rm B} Q^{V}+b}q^{-\frac{1}{2}(\ell,{\rm B}^{-1}\ell)}\prod_{v\in V}\prod_{1\leq j\leq N-1}x_{vj}^{-(\omega_j,\ell_{v})},
\end{gather*}
where $x_{j} = \frac{y_{j}}{y_{j+1}}$.
\end{Definition}
Here ``v.p.'' denotes the principal value integral. That is, taking the average over $W$ number of deformed contours, each corresponding to a Weyl chamber. For instance, for $G={\rm SU}(N)$, the deformed contour corresponding to a permutation $\sigma\in W = S_{N}$ is given by
\begin{gather*}
|y_{\sigma(1)}| < |y_{\sigma(2)}| < \cdots < |y_{\sigma(N)}|. 
\end{gather*}
In practice, this means that there are $W$ number of ways to expand the integrand into power series, and the contour integral simply picks out the constant term in the average of these power series.

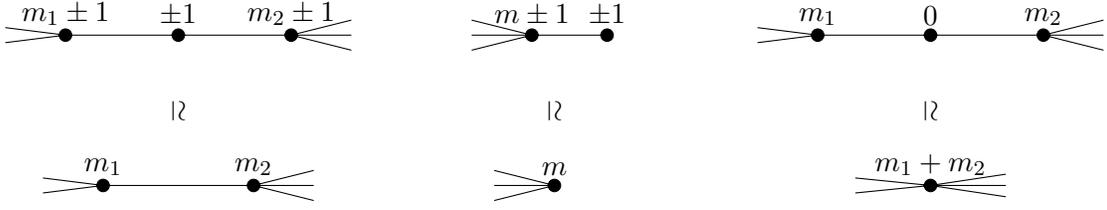
\begin{figure}[t]
 \centering
 \begin{tikzpicture}
 \draw[fill]
 (-1.5,0) node[above]{$m_1 \pm 1$} circle(0.5ex)--
 (0,0) node[above]{$\pm 1$} circle(0.5ex)--
 (1.5,0) node[above]{$m_2 \pm 1$} circle(0.5ex)
 (1.5,0) node[above]{}--(2.3,0.2) node[above]{}
 (1.5,0) node[above]{}--(2.3,0) node[above]{}
 (1.5,0) node[above]{}--(2.3,-0.2) node[above]{}
 (-2.3,0.1) node[above]{}--(-1.5,0) node[above]{}
 (-2.3,-0.1) node[above]{}--(-1.5,-0) node[above]{}
 (0,-1) node[rotate=270]{$\simeq$}
 (-1,-2) node[above]{$m_1$} circle(0.5ex)--
 (1,-2) node[above]{$m_2$} circle(0.5ex)
 (1,-2) node[above]{}--(1.8,-1.8) node[above]{}
 (1,-2) node[above]{}--(1.8,-2) node[above]{}
 (1,-2) node[above]{}--(1.8,-2.2) node[above]{}
 (-1.8,-1.9) node[above]{}--(-1,-2) node[above]{}
 (-1.8,-2.1) node[above]{}--(-1,-2) node[above]{};
 \draw[fill]
 (3.9,0) node[above]{} --
 (4.7,0) node[above]{$m \pm 1$} circle(0.5ex)--
 (5.7,0) node[above]{$\pm 1$} circle(0.5ex)
 (3.9,0.2) node[above]{}--(4.7,0) node[above]{}
 (3.9,-0.2) node[above]{}--(4.7,0) node[above]{}
 (5,-1) node[rotate=270]{$\simeq$}
 (4.2,-2) node[above]{} --
 (5,-2) node[above]{$m$} circle(0.5ex)
 (4.2,-1.8) node[above]{}--(5,-2) node[above]{}
 (4.2,-2.2) node[above]{}--(5,-2) node[above]{};
 \draw[fill]
 (8.5,0) node[above]{$m_1$} circle(0.5ex)--
 (10,0) node[above]{$0$} circle(0.5ex)--
 (11.5,0) node[above]{$m_2$} circle(0.5ex)
 (11.5,0) node[above]{}--(12.3,0.2) node[above]{}
 (11.5,0) node[above]{}--(12.3,0) node[above]{}
 (11.5,0) node[above]{}--(12.3,-0.2) node[above]{}
 (7.7,0.1) node[above]{}--(8.5,0) node[above]{}
 (7.7,-0.1) node[above]{}--(8.5,-0) node[above]{}
 (10,-1) node[rotate=270]{$\simeq$}
 (10,-2) node[above]{$m_1+m_2$} circle(0.5ex)
 (10,-2) node[above]{}--(11,-1.85) node[above]{}
 (10,-2) node[above]{}--(11,-2) node[above]{}
 (10,-2) node[above]{}--(11,-2.15) node[above]{}
 (9,-1.9) node[above]{}--(10,-2) node[above]{}
 (9,-2.1) node[above]{}--(10,-2) node[above]{};
 \end{tikzpicture}
 \caption{Neumann moves on plumbing trees.}\label{Neumann moves}
 \end{figure}

It is known by Neumann that two plumbed manifolds $Y(\Gamma)$ and $Y(\Gamma')$ are the same if and only if the plumbing graphs $\Gamma$ and $\Gamma'$ are related by a sequence of \textit{Neumann moves} in Fig.~\ref{Neumann moves}.
Therefore the following result implies that $\hat{Z}_b^G$ is a topological invariant:
\begin{Theorem}\label{thm:Neumanninvar}
The $q$-series $\hat{Z}_b^G$ defined above is invariant under Neumann moves.
\end{Theorem}
The proof is straightforward and is basically the same as the proof of Proposition 4.6 in \cite{GM}. So here we only give a sketch of the proof.
\begin{proof}Consider the first move. Under this move $3\sigma-\operatorname{Tr} {\rm B}$ remains unchanged. When the signs on the top are $-1$, $\pi$ does not change, and the contribution of the vector $\ell = \big(\vec{\ell}_l,0,\vec{\ell}_r\big)$ for the top graph to the theta function is the same as the contribution of the vector $\ell' = \big(\vec{\ell}_l,\vec{\ell}_r\big)$ for the bottom graph. That is,
\begin{gather}\label{Neumannproof}
 \big(\ell,{\rm B}^{-1}\ell\big) = \big(\ell',{\rm B}'^{-1}\ell'\big).
\end{gather}
When the signs on the top are $+1$, the sign of $(-1)^\pi$ changes, and the contribution of the vector $\ell = \big(\vec{\ell}_l,0,\vec{\ell}_r\big)$ for the top graph is the same as the contribution of the vector $\ell' = \big(\vec{\ell}_l,-\vec{\ell}_r\big)$ for the bottom graph in a sense of~(\ref{Neumannproof}). The effect of change of sign in~$\vec{\ell}_r$ is compensated by the change of sign in $(-1)^{|\Delta^+|\pi}$. Hence~$\hat{Z}^G$ is invariant under the first Neumann move.

Consider the second move. When the signs on the top are $-1$, $\pi$ does not change, and $3\sigma-\operatorname{Tr} {\rm B}$ increases by $1$. The contribution of the vector $\ell = \big(\vec{\ell}_l,\ell_0,w(\rho)\big)$ for the top graph is related to that of the vector $\ell' = \big(\vec{\ell}_l,\ell_0+w(\rho)\big)$ for the bottom graph via
\[\big(\ell,{\rm B}^{-1}\ell\big) = \big(\ell',{\rm B}'^{-1}\ell'\big)-(\rho,\rho).\]
The extra factor of $q^{-\frac{(\rho,\rho)}{2}}$ due to this change is cancelled out by the change in $q^{\frac{3\sigma-\operatorname{Tr} {\rm B}}{2}(\rho,\rho)}$. When the signs on the top are $+1$, both $\pi$ and $3\sigma-\operatorname{Tr} {\rm B}$ decrease by $1$. The contribution of the vector $\ell = \big(\vec{\ell}_l,\ell_0,w(\rho)\big)$ for the top graph is related to that of the vector $\ell' = \big(\vec{\ell}_l,\ell_0-w(\rho)\big)$ for the bottom graph via
\[\big(\ell,{\rm B}^{-1}\ell\big) = \big(\ell',{\rm B}'^{-1}\ell'\big)+(\rho,\rho).\]
The extra factor of $q^{\frac{(\rho,\rho)}{2}}$ due to this change is cancelled by the change in $q^{\frac{3\sigma-\operatorname{Tr} {\rm B}}{2}(\rho,\rho)}$, and the effect of change of sign in $w(\rho)$ is compensated by the change of sign in $(-1)^{|\Delta^+|\pi}$. Hence $\hat{Z}^G$ is invariant under the second Neumann move.

Consider the third move. This time $3\sigma-\operatorname{Tr} {\rm B}$ is preserved, and the sign of $(-1)^\pi$ changes. Under this move, the contribution of the vector $\ell = \big(\vec{\ell}_l,\ell_l,0,\ell_r,\vec{\ell}_r\big)$ for the top graph is the same as the contribution of the vector $\ell' = \big(\vec{\ell}_l,\ell_l-\ell_r,-\vec{\ell}_r\big)$ for the bottom graph, in a sense of~(\ref{Neumannproof}). Again, the effect of change of sign in~$\vec{\ell}_r$ is cancelled out the change of sign in~$(-1)^{|\Delta^+|\pi}$. Hence~$\hat{Z}^G$ is invariant under the third Neumann move.
\end{proof}

\subsection{Some examples and higher rank false theta functions}
\subsubsection[$Y=S_0^3(K_n)$]{$\boldsymbol{Y=S_0^3(K_n)}$}

 The 0-surgery on twist knots are probably the simplest examples. They admit simple plumbing diagrams given in Fig.~\ref{double twist knot figure}.\footnote{Although we have assumed for simplicity in Definition~\ref{def:Zhatmaindef} that the plumbing graph is a tree, we can extend this definition to plumbings with loops, as in~\cite{CGPS}.} For instance,
\begin{table}[h!]\centering
\begin{tabular}{@{\,}c@{\,}|@{\,}c@{\,}}
$G$ & $\hat{Z}_0\big(S_0^3(\mathbf{5}_2)\big)\cong$\bsep{2pt}\\
\hline
${\rm SU}(2)$ & $\frac{1}{2!}\big(1 -q +q^3 -q^6 +q^{10} -q^{15} +q^{21} -q^{28} +q^{36} -q^{45} +q^{55} -q^{66} +q^{78} -\cdots\big)$\tsep{2pt}\bsep{2pt}\\
${\rm SU}(3)$ & $\frac{1}{3!}\big(1 -2q +2q^3 +q^4 -4q^6 +2q^9 +2q^{10} +q^{12} -2q^{13} -4q^{15} +2q^{18} +2q^{19} +\cdots\big)$\tsep{2pt}\bsep{2pt}\\
${\rm SU}(4)$ & $\frac{1}{4!}\big(1 -3q +q^2 +4q^3 -2q^4 +q^5 -5q^6 -2q^7 +3q^8 +2q^9 +9q^{10} -2q^{11} -\cdots\big)$\tsep{2pt}\bsep{2pt}\\
${\rm SU}(5)$ & $\frac{1}{5!}\big(1 -4q +3q^2 +6q^3 -7q^4 -2q^5 +2q^7 -2q^8 +6q^9 +15q^{10} -12q^{11} -23q^{12} +\cdots\big)$
\end{tabular}
\end{table}

Indeed, for every positive twist knot $K_p$ the following is easy to deduce from (\ref{integralZhat}).
\begin{Proposition}\label{twist knot prop}
For the $0$-surgery on the twist knot $K_p$, its $\hat{Z}^G$ is given by
\begin{gather}\label{twist knot Zhat}
\hat{Z}_0^G(S_0^3(K_p)) \cong \frac{1}{|W|}\sum_{\ell\in P_+\cap (Q+\rho)} N_\ell \sum_{w\in W}(-1)^{l(w)}q^{\frac{1}{2}||\sqrt{p}\ell - \frac{1}{\sqrt{p}}w(\rho)||^2} =: \frac{1}{|W|}\chi_{p,\rho},
\end{gather}
where
\begin{gather*}
N_\ell := \sum_{w\in W}(-1)^{l(w)}K(w(\ell)),
\end{gather*}
and $K(\beta)$ denotes the Kostant partition function.\footnote{For example, $N_\ell$ is $ \operatorname{sgn}((\ell,\alpha_1))$ for ${\rm SU}(2)$, and $\operatorname{sgn}\big(\prod\limits_{\alpha\in \Delta_+}(\ell,\alpha)\big) \min\{|(\ell,\alpha_1)|,|(\ell,\alpha_2)|\}$ for ${\rm SU}(3)$.}
\end{Proposition}

Note that $\chi_{p,\rho}$ is exactly the higher rank false theta function (a character of the log-VOA $W^0(p)_Q$) given in equation~(1.2) of~\cite{BM}!
Similarly for double twist knots $K_{m,n}$ with $m,n>0$,\footnote{In our notation, $K_{m,n}$ denotes the double twist knot with $m$ and $n$ full twists.}
\begin{equation}\label{double twist knot Zhat}
\hat{Z}_0^G(S_0^3(K_{m,n})) \cong \frac{1}{|W|}\chi_{m,\rho}\chi_{n,\rho}.
\end{equation}
\begin{proof}[Proof of Proposition \ref{twist knot prop}]
 The 0-surgery on $K_p$ has a simple plumbing description as shown in Fig.~\ref{double twist knot figure}.
 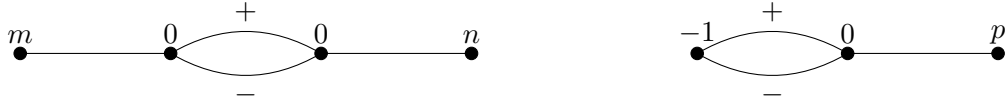
\begin{figure}[t]
 \centering
 \begin{tikzpicture}
 \draw[fill]
 (-2,0) node[above]{$m$} circle(0.5ex)--
 (0,0) node[above]{$0$} circle(0.5ex)
 (2,0) node[above]{$0$} circle(0.5ex)--
 (4,0) node[above]{$n$} circle(0.5ex);
 \draw (0,0) edge[bend left] node[midway,above]{$+$} (2,0);
 \draw (0,0) edge[bend right] node[midway,below]{$-$} (2,0);

 \draw[fill]
 (7,0) node[above]{$-1$} circle(0.5ex)
 (9,0) node[above]{$0$} circle(0.5ex)--
 (11,0) node[above]{$p$} circle(0.5ex);
 \draw (7,0) edge[bend left] node[midway,above]{$+$} (9,0);
 \draw (7,0) edge[bend right] node[midway,below]{$-$} (9,0);
 \end{tikzpicture}
 \caption{Plumbing diagrams for the 0-surgery on $K_{m,n}$ and $K_{p}=K_{1,p}$.}\label{double twist knot figure}
 \end{figure}
 The linking matrix and its inverse are
 \[{\rm B} = \begin{pmatrix} -1 & 0 & 0 \\ 0 & 0 & 1 \\ 0 & 1 & p\end{pmatrix}
 ,\qquad
 {\rm B}^{-1} = \begin{pmatrix} -1 & 0 & 0 \\ 0 & -p & 1 \\ 0 & 1 & 0\end{pmatrix}.\]
 There is a single trivalent vertex with $0$ framing. This contributes the following factor in $F_{3d}(x)$:
 \[ \bigg(\sum_{w\in W}(-1)^{l(w)}x_0^{w(\rho)}\bigg)^{-1} = \frac{1}{|W|}\sum_{\ell_0\in P_+ \cap (Q+\rho)}N_{\ell_0} \sum_{w\in W}(-1)^{l(w)}x_0^{w(\ell_0)}.\]
 For $\ell = (0,\ell_0,\ell_p)^t$,
 \[q^{-\frac{1}{2}(\ell,{\rm B}^{-1}\ell)} = q^{\frac{1}{2}||\sqrt{p}\ell_0 - \frac{1}{\sqrt{p}}\ell_p||^2 -\frac{1}{2p}||\ell_p||^2}.\]
 Applying (\ref{integralZhat}), it is straightforward to get~(\ref{twist knot Zhat}).

 Using a plumbing description of the 0-surgery on $K_{m,n}$ (Fig.~\ref{double twist knot figure}), it is easy to derive~(\ref{double twist knot Zhat}) as well.
\end{proof}

\subsubsection[$Y=\Sigma(p_1,p_2,p_3)$]{$\boldsymbol{Y=\Sigma(p_1,p_2,p_3)}$}

 For convenience let us define the following notation for \textit{higher rank false theta functions}:
\begin{gather*}
\chi_{p,\beta}^{G} := \sum_{\ell\in P_+\cap (Q+\rho)} N_\ell \sum_{w\in W}(-1)^{l(w)}q^{\frac{1}{2}||\sqrt{p}\ell - \frac{1}{\sqrt{p}}w(\beta)||^2}.
\end{gather*}
When $G={\rm SU}(2)$, this becomes
\begin{gather*}
\chi_{p,n\rho}^{{\rm SU}(2)} = \Psi_{p,p-n},\qquad \text{for} \quad n = 1, \dots, p-1,
\end{gather*}
where
\begin{gather*}
\Psi_{p,r}:= \sum_{\substack{\ell\in \mathbb{Z}\\ \ell = r\text{ mod } 2p}}\operatorname{sgn}(\ell) q^{\ell^2/4p}
\end{gather*}
is the usual false theta function, and in this sense $\chi^G_{p,\beta}$ is the higher rank generalization of the false theta functions.
\begin{Proposition}\label{Brieskorn prop}
For the Brieskorn sphere $Y = \Sigma(p_1,p_2,p_3)$ with $0 < p_1 < p_2 < p_3$ pairwise relatively prime, we have
\begin{gather*}
\hat{Z}_0^G(\Sigma(p_1,p_2,p_3)) \cong \sum_{(w_1,w_2)\in W^2}(-1)^{l(w_1w_2)}\chi_{p_1p_2p_3,p_2p_3\rho + p_1p_3w_1(\rho) + p_1p_2w_2(\rho)}.
\end{gather*}
\end{Proposition}
That is, it is a sum of $|W|^2$ number of higher rank false theta functions.\footnote{That $\hat{Z}$'s for Brieskorn spheres should be expressed as sums of higher rank false theta functions was envisaged earlier in \cite{CCFGH}.}
\begin{proof}
The proof is analogous to that of Proposition 4.8 in \cite{GM}.
\end{proof}
Note that we did not have to treat $\Sigma(2,3,5)$ separately. In this sense, using $\chi_{p,\beta}$ as false theta functions is more natural than using $\Psi_{p,n}$.

\subsubsection[$Y = M\big(a_0;\frac{a_1}{b_1},\frac{a_2}{b_2},\frac{a_3}{b_3}\big)$]{$\boldsymbol{Y = M\big(a_0;\frac{a_1}{b_1},\frac{a_2}{b_2},\frac{a_3}{b_3}\big)}$}

 Let $b_1,b_2,b_3>0$ and assume that $Y$ has negative orbifold number, i.e.,
\begin{equation}
e = a_0 + \sum_{j=1}^{3}\frac{a_j}{b_j} < 0.
\end{equation}
Assume further that the central meridian is trivial in homology, i.e.,
\begin{gather*}
e\operatorname{lcm}(b_1,b_2,b_3) = -1.
\end{gather*}
Then their $\hat{Z}_b$'s can be expressed as signed sum of higher rank false theta functions:
\begin{Proposition}\label{some Seifert prop}
Under the assumptions as above, $\hat{Z}^G$ for $Y=M\big(a_0;\frac{a_1}{b_1},\frac{a_2}{b_2},\frac{a_3}{b_3}\big)$ is given by
\begin{gather}
\hat{Z}_b^G\left(M\left(a_0;\frac{a_1}{b_1},\frac{a_2}{b_2},\frac{a_3}{b_3}\right)\right) \nonumber\\
\qquad{} \cong \sum_{(w_1,w_2)\in W^2}\mathbf{1}_b(w_1,w_2) (-1)^{l(w_1w_2)}\chi_{\frac{b_1b_2b_3}{|H_1|},\frac{b_2b_3}{|H_1|}\rho + \frac{b_1b_3}{|H_1|}w_1(\rho) + \frac{b_1b_2}{|H_1|}w_2(\rho)},\label{some Seifert Zhat}
\end{gather}
where
\begin{gather*}
\mathbf{1}_b(w_1,w_2) := \begin{cases} 1, &\text{if } \ell(\rho,\rho,w_1(\rho),w_2(\rho)) \in {\rm B}Q^V+b, \\ 0, &\text{otherwise.}\end{cases}
\end{gather*}
\end{Proposition}
Observe that Proposition \ref{some Seifert prop} is a slight generalization of Proposition \ref{Brieskorn prop}.
\begin{proof}
 Since $Y$ is a Seifert manifold with $3$ singular fibers, it can be described as a star-shaped plumbing with 3 legs. The only vertices whose degree is not 2 are the central vertex and the terminal vertices. Denote by $\ell(\ell_0,\ell_1,\ell_2,\ell_3)$ an element $\ell\in {\rm B}Q^V + b$ such that
 \[\ell_v = \begin{cases} \ell_0, & v \text{ is the central vertex},\\ \ell_1, \ell_2, \ell_3, & v \text{ is the corresponding terminal vertex}, \\ 0, & \text{otherwise}.\end{cases}\]
 Then for any $\ell$ with $\ell_1,\ell_2,\ell_3 \in W(\rho)$,
 \[q^{-\frac{1}{2}(\ell,{\rm B}^{-1}\ell)} = q^{\frac{1}{2|H_1|}||\sqrt{b_1b_2b_3}\ell_0 - \frac{1}{\sqrt{b_1b_2b_3}}(b_2b_3\ell_1 + b_3b_1\ell_2 + b_1b_2\ell_3)||^2 + C}\]
 for some constant $C$ independent of $\ell$. Applying (\ref{integralZhat}), it is straightforward to obtain (\ref{some Seifert Zhat}). Note that the assumption $e\operatorname{lcm}(b_1,b_2,b_3)=-1$ was introduced so that
 \begin{gather*} \ell(\rho,\rho,w_1(\rho),w_2(\rho))\in {\rm B}Q^V+b \Leftrightarrow \ell(\rho + Q,\rho,w_1(\rho),w_2(\rho))\in {\rm B}Q^V+b.\tag*{\qed}\end{gather*}\renewcommand{\qed}{}
\end{proof}

We end this section by posing some interesting open questions arising from our definition of higher rank $\hat{Z}$.
When $G={\rm SU}(2)$, the study of modular-like properties of the $q$-series $\hat{Z}$ was initiated in \cite{BMM1, BMM2, CCFGH}. It is proved in special cases and conjectured for general negative definite plumbings that, although $\hat{Z}$ is not modular in the traditional sense, it exhibits a more exotic modular property called \emph{quantum modularity} \`a la Zagier~\cite{Z}.
Naturally, we can ask if the same is true for higher rank~$\hat{Z}$.
\begin{Question}
 What are the quantum modular properties of $\hat{Z}^G$?
\end{Question}
Another open question is to extend our definition to drop the condition of (weakly) negative definiteness. In particular, the orientation reversal of a negative definite plumbed manifold is positive definite. While for polynomial invariants, such as Jones polynomials, the orientation reversal is governed simply by the transformation $q\leftrightarrow q^{-1}$, for $q$-series invariants such transformation is ill-defined, and the orientation reversal is more subtle. This issue, in $G={\rm SU}(2)$ case, was first tackled in~\cite[Section~7]{CCFGH}, where it is observed that for some negative definite plumbed manifolds $Y$, the $\hat{Z}$ of the orientation reversal~$-Y$ can be obtained by taking the mock theta function associated to the false theta function~$\hat{Z}(Y)$. It would be very interesting to study this orientation reversal problem in our higher rank setting.
\begin{Question}
 How does $\hat{Z}^G$ behave under orientation reversal? What are the higher rank mock theta functions associated to the higher rank false theta functions?
\end{Question}

\section[Higher rank $F_K$]{Higher rank $\boldsymbol{F_K}$}\label{sec:higherrankFK}
\subsection[Review of ${\rm SU}(2)$ case]{Review of $\boldsymbol{{\rm SU}(2)}$ case}

Recently Gukov and Manolescu \cite{GM} conjectured the existence of a knot invariant $F_K(x,q)$, which is the analogue of $\hat{Z}$ for knot complements.
\begin{Conjecture}[Gukov--Manolescu \cite{GM}]\label{conj:SU2FK}
 For any knot $K$, there exists a series
 \begin{gather*}
 F_K(x,q) = \frac{1}{2!}\sum_{\substack{m\geq1\\\text{odd}}} \big(x^{m/2}-x^{-m/2}\big)f_m(q),
 \end{gather*}
 whose coefficients $f_m(q)$ are Laurent series with integer coefficients, such that the asymptotic expansion of $F_K\big(x,{\rm e}^\hbar\big)$ is the same as the Melvin--Morton--Rozansky expansion~{\rm \cite{BG,MM,R}} for colored Jones polynomials:
 \begin{gather*}
 F_K\big(x,{\rm e}^\hbar\big) = \big(x^{1/2}-x^{-1/2}\big)\sum_{j\geq 0}\frac{P_j(x)}{\Delta_K(x)^{2j+1}}\frac{\hbar^j}{j!},
 \end{gather*}
 where $P_j(x)\in \mathbb{Z}\big[x,x^{-1}\big]$, $P_0 = 1$, and $\Delta_K(x) = \nabla_K\big(x^{1/2}-x^{-1/2}\big)$ is the Alexander polynomial for $K$. In particular, in the semi-classical limit we should have
 \begin{gather*}
 \lim_{q\rightarrow 1}F_K(x,q) = \frac{x^{1/2}-x^{-1/2}}{\Delta_K(x)}.
 \end{gather*}
 Moreover, this series is annihilated by the quantum $A$-polynomial:
 \begin{gather*}
 \hat{A}_K(\hat{x},\hat{y},q)F_K(x,q) = 0.
 \end{gather*}
\end{Conjecture}
In \cite{GM}, $F_K$ is computed for torus knots and the figure-8 knot, and more recently, in~\cite{P}, $F_K$ is computed for a bigger class of knots, including positive braid knots, positive double twist knots and the Whitehead link.

\subsection[Higher rank $F_K$]{Higher rank $\boldsymbol{F_K}$}

Let us study the higher rank generalization of $F_K(x,q)$. As a natural generalization of Conjecture~\ref{conj:SU2FK}, we make the following conjecture.
\begin{Conjecture}[higher rank $F_K$]
For any knot $K$ and a choice of a root system $G$, there exists a series
\begin{gather*}
F_{K}^G(\mathbf{x},q) = \frac{1}{|W|}\sum_{\beta\in P_+\cap (Q+\rho)}f_\beta^G(q)\sum_{w\in W}(-1)^{l(w)}x^{w(\beta)},
\end{gather*}
where $\mathbf{x} = (x_1,\dots,x_r)$ and the coefficients $f_\beta^G(q)$ are Laurent series with integer coefficients, such that its asymptotic expansion agrees with the higher rank Melvin--Morton--Rozansky expansion for the higher rank colored Jones polynomials
\begin{gather*}
 F_{K}^G\big(\mathbf{x},{\rm e}^\hbar\big) = \prod_{\alpha\in \Delta^+}\big(x^{\alpha/2}-x^{-\alpha/2}\big)\sum_{j\geq 0}\frac{P_j(\mathbf{x})}{\Big(\prod\limits_{\alpha\in \Delta^+} \Delta_K(x^\alpha)\Big)^{2j+1}}\frac{\hbar^j}{j!},
\end{gather*}
where $P_j(\mathbf{x})\in \mathbb{Z}\big[x_1,x_1^{-1},\dots,x_r,x_r^{-1}\big]$ and $P_0 = 1$. In particular, in the semi-classical limit we should have
\begin{gather*}
 \lim_{q\rightarrow 1}F_K^G(\mathbf{x},q) = \prod_{\alpha\in \Delta^+}\frac{x^{\alpha/2}-x^{-\alpha/2}}{\Delta_K(x^\alpha)}.
\end{gather*}
Moreover, this series should be annihilated by the $($higher rank$)$ quantum $A$-polynomial:
\begin{gather*}
 \hat{A}_K\big(\hat{x}_1,\hat{y}_1,\dots,\hat{x}_r,\hat{y}_r\big)F_K^G(\mathbf{x},q) = 0.
\end{gather*}
\end{Conjecture}

Our main result in this section is an explicit expression for $F_K^G(\mathbf{x},q)$ for torus knots.
\begin{Theorem}[higher rank $F_K$ for torus knots]\label{thm:FKfortorusknots}
For $K = T_{s,t}$, $f_\beta^G(q)$ is a monomial of deg\-ree~$\frac{(\beta,\beta)}{2st}$, up to an overall $q$-power. More precisely,
\begin{gather}
F_{T_{s,t}}^G \cong \frac{1}{|W|}\sum_{\beta\in P_+ \cap (Q+\rho)}\sum_{w\in W}(-1)^{l(w)}x^{w(\beta)}\nonumber\\
\hphantom{F_{T_{s,t}}^G \cong}{} \times\sum_{(w_1,w_2)\in W^2}(-1)^{l(w_1w_2)}\mathbf{1}(\beta,w_1,w_2)N_{\frac{1}{st}(\beta + tw_1(\rho)+sw_2(\rho))} q^{\frac{(\beta,\beta)}{2st}},\label{torus knot F_K}
\end{gather}
where
\begin{gather*}
\mathbf{1}(\beta,w_1,w_2):=\begin{cases} 1, &\text{if }\frac{1}{st}(\beta + tw_1(\rho) + sw_2(\rho)) \in P_+\cap (Q+\rho),\\ 0,&\text{otherwise}.\end{cases}
\end{gather*}
\end{Theorem}
\begin{proof}This can be derived either directly from (\ref{integralZhat}) by using plumbing description or by reverse-engineering using the higher rank surgery formula that we discuss below.
 Here we present a direct derivation. Recall from~\cite{GM} that the complement of $T_{s,t}$ has a plumbing description as in Fig.~\ref{torus knot figure}, where $0<t'<t$, $0<s'<s$ are chosen such that $st'\equiv -1~(\text{mod }t)$ and $ts' \equiv -1~(\text{mod }s)$.
 \begin{figure}[t] \centering
 \begin{tikzpicture}
 \draw[fill]
 (-2,0) node[above]{$-\frac{t}{t'}$} circle(0.5ex)--
 (0,0) node[above]{$-1$} circle(0.5ex) --
 (2,0) node[above]{$-\frac{s}{s'}$} circle(0.5ex);
 \draw (0,-2) node[below]{$-st$} circle(0.5ex)
 (0,-1.91) node[above]{}--(0,0);
 \end{tikzpicture}
 \caption{Complement of $T_{s,t}$.} \label{torus knot figure}
 \end{figure}
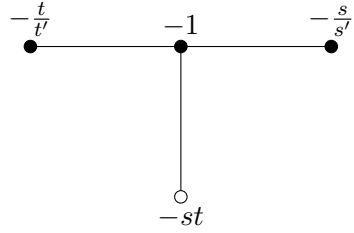
 The linking matrix is
 \[{\rm B} = \begin{pmatrix} -st & 1 & 0 & 0 \\ 1 & -1 & 1 & 1 \\ 0 & 1 & \big({-}\frac{t}{t'}\big) & 0 \\ 0 & 1 & 0 & \big({-}\frac{s}{s'}\big)\end{pmatrix},\]
 where $\big({-}\frac{t}{t'}\big)$ and $\big({-}\frac{s}{s'}\big)$ should be understood as block matrices corresponding to the continued fractions.
 To compute the integral~(\ref{integralZhat}) with $x_{-st}$ left unintegrated, we just have to replace the theta function $\Theta^{-{\rm B}}\big(x^{-1},q\big)$ with
 \[\Theta^{-{\rm B}'}\big(x^{-1},q\big) \cong \sum_{\alpha\in Q^{V'}}q^{-\frac{1}{2}(\alpha,{\rm B}'\alpha)-(\alpha,\delta)}\prod_{v\in V'}x_v^{-({\rm B}'\alpha + \delta)}\cdot x_{-st}^{-\alpha_{-1}-\rho},\]
 where $V' = V \setminus \{v_{-st}\}$ and ${\rm B}'$ is the corresponding sub-linking matrix. Set $\beta = -\alpha_{-1}-\rho$.
 We need to multiply $\Theta^{-{\rm B}'}\big(x^{-1},q\big)$ with
 \[\prod_{v\in V'}\bigg(\sum_{w\in W}(-1)^{l(w)}x_v^{w(\rho)}\bigg)^{2-\deg v}\]
and take the constant term with respect to variables $x_v$, $v\in V'$. As $2-\deg v$ is non-zero for only 3 vertices (the central vertex $v_{-1}$ and the 2 terminal vertices) it is pretty easy to compute. The only contributions come from those $\alpha$'s such that ${\rm B}'\alpha + \delta$ takes values $w_1(\rho)$, $w_2(\rho)$ on the terminal vertices for some $w_1,w_2\in W$, a value in $Q+\rho$ in the central vertex, and $0$ on all the other vertices. Using simple linear algebra, it is easy to check that for those $\alpha$'s,
 \[q^{-\frac{1}{2}(\alpha,{\rm B}'\alpha) - (\alpha,\delta)} = q^{\frac{(\beta,\beta)}{2st}+C}\]
for some constant $C$ independent of $\alpha$, and that $\frac{1}{st}(\beta + tw_1(\rho) + sw_2(\rho))$ is the value of ${\rm B}'\alpha + \delta$ on the central vertex. This proves~(\ref{torus knot F_K}).
\end{proof}

\begin{Example}[right-handed trefoil with $G = {\rm SU}(3)$] The first few $f_\beta^{{\rm SU}(3)}$ are (up to overall sign and $q$-power)
\begin{gather*}
f_{(1,1)} = -q,\qquad
f_{(4,1)} = -q^2,\qquad
f_{(5,2)} = -2q^3,\qquad
f_{(7,1)} = -q^4,\qquad
f_{(5,5)} = q^5, \nonumber\\
f_{(7,4)} = 2q^6,\qquad
f_{(10,1)} = -q^7,\qquad
f_{(8,5)} = q^8,\qquad
f_{(11,2)} = -2q^9,\qquad
f_{(7,7)} = -q^9, \nonumber\\
f_{(13,1)} = -q^{11},\qquad
f_{(11,5)} = q^{12},\qquad
\dots,
\end{gather*}
where we have written $\beta$ in the fundamental weights basis. (Because $f_{(m,n)} = f_{(n,m)}$, we have only written those terms with $m \geq n$.) The $q$-power of this $f_\beta$ is, up to overall constant,
\begin{gather*}
\frac{(\beta,\beta)}{12}.
\end{gather*}
In the $q\rightarrow 1$ limit we have, as expected,
\begin{gather*}
F_{\mathbf{3}_1^r}^{{\rm SU}(3)}(x_1,x_2,1) = \frac{x_1^{1/2}-x_1^{-1/2}}{x_1+x_1^{-1}-1}\frac{x_2^{1/2}-x_2^{-1/2}}{x_2+x_2^{-1}-1}\frac{x_1^{1/2}x_2^{1/2}-x_1^{-1/2}x_2^{-1/2}}{x_1x_2+x_1^{-1}x_2^{-1}-1}.\label{classicallimit}
\end{gather*}
\end{Example}

Just as in ${\rm SU}(2)$ case \cite{GM}, we conjecture the following surgery formula (analogous to Conjecture~1.7 of~\cite{GM}) relating $F_K^G$ to $\hat{Z}^G_b\big(S_{p/r}^3(K)\big)$:
\begin{Conjecture}[higher rank surgery formula]
Let $K\subset S^3$ be a knot. Then
\begin{gather*}
\hat{Z}_b^G\big(S_{p/r}^3(K)\big) \cong \mathcal{L}_{p/r}^{(b)}\left[\prod_{\alpha\in \Delta^+}\big(x^{\frac{\alpha}{2r}}-x^{-\frac{\alpha}{2r}}\big)F_K^G(\mathbf{x},q) \right],
\end{gather*}
whenever the RHS makes sense.
\end{Conjecture}

This is a theorem for knots and 3-manifolds represented by negative-definite plumbings, as a~straightforward generalization of Theorem~1.2 of~\cite{GM}.
For instance, surgery on $\mathbf{3}_1^r$ gives us the following $\hat{Z}^{{\rm SU}(3)}$'s:
\begin{table}[h!]\centering
\begin{tabular}{@{\,}c@{\,}|@{\,}c@{\,}|@{\,}c@{\,}}
$r$ & $S_{-1/r}^3(\mathbf{3}_1^r)$ & $\hat{Z}_0^{{\rm SU}(3)}\big(S_{-1/r}^3(\mathbf{3}_1^r)\big)$\bsep{2pt}\\
\hline
$1$ & $\Sigma(2,3,7)$ & $1 -2q +2q^3 +q^4 -2q^5 -2q^8 + 4q^9 + 2q^{10} - 4q^{11} +2q^{13} -6q^{14} +\cdots$\tsep{2pt}\bsep{2pt}\\
$2$ & $\Sigma(2,3,13)$ & $1 -2q +2q^3 -q^4 +2q^{10} -2q^{11} -2q^{14} +2q^{16} +2q^{19} -2q^{20} +4q^{21} -\cdots$\tsep{2pt}\bsep{2pt}\\
$3$ & $\Sigma(2,3,19)$ & $1 -2q +2q^3 -q^4 +2q^{16} -2q^{17} -2q^{20} +2q^{22} +2q^{25} -2q^{26} +4q^{33} -\cdots$\tsep{2pt}\bsep{2pt}\\
$4$ & $\Sigma(2,3,25)$ & $1 -2q +2q^3 -q^4 +2q^{22} -2q^{23} -2q^{26} +2q^{28} +2q^{31} -2q^{32} +4q^{45} -\cdots$\tsep{2pt}\bsep{2pt}\\
$5$ & $\Sigma(2,3,31)$ & $1 -2q +2q^3 -q^4 +2q^{28} -2q^{29} -2q^{32} +2q^{34} +2q^{37} -2q^{38} +4q^{57} -\cdots$\tsep{2pt}\bsep{2pt}\\
$r$ & $\Sigma(2,3,6r+1)$ & $\sum\limits_{(w_1,w_2)\in W^2}(-1)^{l(w_1w_2)}\chi_{36r+6, 3(6r+1)w_1(\rho)+ 2(6r+1)w_2(\rho)+ 6\rho}$
\end{tabular}
\end{table}

In fact it is easy to check that for $K=T_{s,t}$,
\begin{gather*}
 \mathcal{L}_{-1/r}\left[\prod_{\alpha\in \Delta^+}\big(x^{\frac{\alpha}{2r}}-x^{-\frac{\alpha}{2r}}\big)F_K(\mathbf{x},q) \right]\\
\qquad {} \cong \sum_{(w_1,w_2)\in W^2}(-1)^{l(w_1w_2)}\chi_{st(rst+1),t(rst+1)w_1(\rho) + s(rst+1)w_2(\rho) +st\rho}\\
\qquad{} \cong \hat{Z}_0^G(\Sigma(s,t,rst+1)),
\end{gather*}
which is consistent with what we have seen in Proposition \ref{Brieskorn prop}.

Furthermore we conjecture that our 0-surgery formula in \cite{CGPS} holds for higher rank as well:
\begin{Conjecture}[higher rank 0-surgery formula]
Let $K\subset S^3$ be a knot. Then
\begin{gather*}
\hat{Z}_0^G\big(S_{0}^3(K)\big) \cong \frac{1}{|W|}f_\rho^G(K).
\end{gather*}
In particular, for positive double twist knots $K_{m,n}$,
\begin{gather*}
 f_\rho^G(K_{m,n}) \cong \chi_{m,\rho}\chi_{n,\rho}.
\end{gather*}
\end{Conjecture}

\section[Symmetric representations and large $N$]{Symmetric representations and large $\boldsymbol{N}$}\label{sec:largeN}
\subsection{Specialization to symmetric representations}

In this section we study a specialization of $F_K^G({\bf x},q)$ to symmetric representations. We restrict our attention to $G={\rm SU}(N)$. We start from the reduced version of $F_K$:
\begin{gather*}
F_K^{\rm red}(\mathbf{x},q) := \frac{1}{|W|}\sum_{\beta\in P_+ \cap (Q+\rho)}f_\beta(q)\frac{\sum\limits_{w\in W}(-1)^{l(w)}x^{w(\beta)}}{\sum\limits_{w\in W}(-1)^{l(w)}x^{w(\rho)}}.
\end{gather*}
Then the (reduced) symmetrically colored $F_K$ corresponds to the following specialization:
\begin{gather*}
F_K^{\rm sym}(x,q) := F_K^{\rm red}((x,q,\dots,q),q).
\end{gather*}
That is, we set $x_2 = \cdots = x_r = q$.
A version of quantum volume conjecture~\cite{FGS} states that this should be annihilated by the symmetrically colored quantum $A$-polynomial:
\begin{gather}\label{QuantVolConj}
\hat{A}_K\big(\hat{x},\hat{y},a=q^N,q\big)F_K^{{\rm SU}(N),{\rm sym}}(x,q) = 0.
\end{gather}

\begin{Example}[right-handed trefoil] For the right-handed trefoil, $F_{\mathbf{3}_1^r}^{{\rm SU}(N),{\rm sym}}(x,q)$ for the first few values of $N$ look like the following:
\begin{itemize}\itemsep=0pt
\item For ${\rm SU}(2)$,
\begin{gather*}
F_{\mathbf{3}_1^r}^{\rm sym}(x,q) \cong \frac{1}{2}\big[ \big({-}q + q^2 + q^3 - q^6 -q^8 + q^{13} + q^{16} - \cdots\big) \\
\hphantom{F_{\mathbf{3}_1^r}^{\rm sym}(x,q) \cong}{} + \big(x + x^{-1}\big)\big(q^2 + q^3 - q^6 -q^8 + q^{13} + q^{16} - \cdots\big)\\
\hphantom{F_{\mathbf{3}_1^r}^{\rm sym}(x,q) \cong}{} + \big(x^2 + x^{-2}\big)\big(q^2 + q^3 - q^6 -q^8 + q^{13} + q^{16} - \cdots\big)\\
\hphantom{F_{\mathbf{3}_1^r}^{\rm sym}(x,q) \cong}{} + \big(x^3 + x^{-3}\big)\big(q^3 - q^6 -q^8 + q^{13} + q^{16} - \cdots\big)\\
\hphantom{F_{\mathbf{3}_1^r}^{\rm sym}(x,q) \cong}{} + \big(x^4 + x^{-4}\big)\big({-}q^6 - q^8 + q^{13} + q^{16} - \cdots\big) + \cdots \big].
\end{gather*}
\item For ${\rm SU}(3)$,
\begin{gather*}
F_{\mathbf{3}_1^r}^{\rm sym}(x,q) \cong \frac{1}{2}\big[\big({-}2q -2q^2 +2q^4 +4q^5 +4q^6 +4q^7 +2q^8 -2q^{10} -4q^{11} -\cdots\big) \\
\hphantom{F_{\mathbf{3}_1^r}^{\rm sym}(x,q) \cong}{} + \big(q^{1/2}x+q^{-1/2}x^{-1}\big)q^{1/2}\big({-}1 -2q -q^2 +q^3 +3q^4 +4q^5 +4q^6 +\cdots\big)\\
\hphantom{F_{\mathbf{3}_1^r}^{\rm sym}(x,q) \cong}{} + \big(qx^2 + q^{-1}x^{-2}\big)\big({-}q -q^2 +2q^4 +3q^5 +4q^6 +3q^7 + 2q^8 -2q^{10} +\cdots\big)\!\\
\hphantom{F_{\mathbf{3}_1^r}^{\rm sym}(x,q) \cong}{} + \big(q^{3/2}x^3 + q^{-3/2}x^{-3}\big)q^{1/2}\big(q^3 +2q^4 +3q^5 +3q^6 +2q^7 +q^8 +\cdots\big)\\
\hphantom{F_{\mathbf{3}_1^r}^{\rm sym}(x,q) \cong}{}+ \big(q^2x^4 + q^{-2}x^{-4}\big)\big(q^3 +q^4 +2q^5 +2q^6 +2q^7 +q^8 +\cdots\big) +\cdots \big].
\end{gather*}
\item For ${\rm SU}(4)$,
\begin{gather*}
F_{\mathbf{3}_1^r}^{\rm sym}(x,q) \cong \frac{1}{2}\big[\big(q^{-2} +q^{-1} -2 -4q -8q^2 -7q^3 -7q^4 +\cdots\big) \\
\hphantom{F_{\mathbf{3}_1^r}^{\rm sym}(x,q) \cong}{}+ \big(qx + q^{-1}x^{-1}\big)\big(q^{-2} -1 -5q -6q^2 -8q^3 -5q^4 -2q^5 +\cdots\big)\\
\hphantom{F_{\mathbf{3}_1^r}^{\rm sym}(x,q) \cong}{}+ \big(q^2x^2 + q^{-2}x^{-2}\big)\big({-}2 -3q -6q^2 -5q^3 -5q^4 +4q^6 +\cdots\big)\\
\hphantom{F_{\mathbf{3}_1^r}^{\rm sym}(x,q) \cong}{} + \big(q^3x^3 + q^{-3}x^{-3}\big)\big({-}q^{-1} -1 -3q -3q^2 -4q^3 -2q^4 +5q^6 +9q^7 +\cdots\!\big)\!\\
\hphantom{F_{\mathbf{3}_1^r}^{\rm sym}(x,q) \cong}{} + \big(q^4x^4 + q^{-4}x^{-4}\big)\big({-}1 -q -2q^2 -q^3 -q^4 +2q^5 +4q^6 +8q^7 +\cdots\big) \\
\hphantom{F_{\mathbf{3}_1^r}^{\rm sym}(x,q) \cong}{}+ \cdots \big].
\end{gather*}
\end{itemize}
Note that the overall factor is $\frac{1}{2}$ instead of $\frac{1}{N!}$. This is due to reduction of the Weyl symmetry to~$\mathbb{Z}_2$ as we specialize to symmetric representations.

It is easy to experimentally check (\ref{QuantVolConj}) term by term in this case, using the $a$-deformed quantum $A$-polynomial for the right-handed trefoil
\[\hat{A}_{\mathbf{3}_1^r}(\hat{x},\hat{y},a,q) = a_0 + a_1\hat{y} + a_2\hat{y}^2,\]
where
\begin{gather*}
a_0 = -\frac{(-1+\hat{x})\big({-}1+aq\hat{x}^2\big)}{a\hat{x}^3(-1+a\hat{x})\big({-}q + a\hat{x}^2\big)},\\
a_1 = \frac{\big({-}1+a\hat{x}^2\big)\big({-}a^2\hat{x}^2 + aq^3\hat{x}^2 +aq\hat{x}(1+\hat{x}+a(-1+\hat{x})\hat{x})-q^2\big(1+a^2\hat{x}^4\big)\big)}{a^2q\hat{x}^3(-1+a\hat{x})\big({-}q + a\hat{x}^2\big)},\\
a_2 = 1
\end{gather*}
with $a$ specialized to $q^N$.
\end{Example}

\subsection[Future direction: large $N$]{Future direction: large $\boldsymbol{N}$}

From (\ref{QuantVolConj}), we are naturally led to the following conjecture:
\begin{Conjecture}[HOMFLY-PT analogue of $F_K$]\label{HOMFLYPTFK}
 For each knot $K$, there exists a function $F_K(x,a,q)$ such that
 \begin{gather*}
 \hat{A}_K(\hat{x},\hat{y},a,q)F_K(x,a,q) = 0
 \end{gather*}
 and
 \begin{gather}\label{SUNspecialization}
 F_K\big(x,q^N,q\big) = F_K^{{\rm SU}(N),{\rm sym}}(x,q).
 \end{gather}
 Moreover, this function should have the following Weyl symmetry:
 \begin{gather*}
 F_{K}\big(x^{-1},a,q\big) = F_{K}\big(a^{-1}q^2x,a,q\big).
\end{gather*}
\end{Conjecture}
In particular, (\ref{SUNspecialization}) implies
\begin{gather*}
\lim_{q\rightarrow 1}F_K\big(x,q^N,q\big) = \Delta_K(x)^{1-N}.
\end{gather*}
The study of this HOMFLY-PT analogue of $F_K$ is the subject of~\cite{EGG+}.

We end with a more speculative question regarding homology theories:
\begin{Question}
 What is $\mathcal{H}_{b,{\rm BPS}}^G$ categorifying $\hat{Z}_b^G(q)$? Is there a family of differentials analogous to that of~{\rm \cite{DGR}}?
\end{Question}
We will pursue these questions in our future work.

\appendix
\section{Motivation for Definition \ref{def:Zhatmaindef}}\label{Gauss}
In this section we motivate Definition \ref{def:Zhatmaindef} through a heuristic decomposition of WRT invariants. This is essentially a generalization of Appendix A of \cite{GPPV}.
Let $Y$ be a plumbing on a connected tree $\Gamma$, each vertex $v$ decorated with an integer $a_v$, with a negative definite linking matrix ${\rm B}$. Then the WRT invariant $Z_{G_k}(Y)$ can be computed by\footnote{Here we ignored a framing factor.}
\begin{gather}
Z_{G_k}(Y) \cong \sum_{\rm colorings}\prod_{v\in V}\mathcal{V}_v\prod_{e\in E}\mathcal{E}_e, \label{eq2.1}
\end{gather}
where the vertex and the edge factors are
\begin{gather*}
\mathcal{V} = t_{\lambda\lambda}^{a_v} s_{\rho\lambda}^{2-\deg v},\qquad
 \mathcal{E} = s_{\mu\lambda}.\end{gather*}
Here the $s$, $t$ matrices are as usual\footnote{See \cite[Theorem~3.3.20]{BK}.}
\begin{gather*}
s_{\lambda \mu} = \frac{{\rm i}^{|\Delta_+|}}{|P/kQ^\vee|^{1/2}}\sum_{w\in W}(-1)^{l(w)}q^{(w(\lambda),\mu)},
\qquad
t_{\lambda \mu} = \delta_{\lambda \mu} q^{\frac{1}{2}(\lambda,\lambda)}q^{-\frac{1}{2}(\rho,\rho)}
\end{gather*}
with $q={\rm e}^{\frac{2\pi{\rm i}}{m k}}$ and set of (shifted) colors being
\[\lambda,\mu\in C = \big\{\lambda\in P_+ + \rho \,\vert\, \big(\lambda,\theta^\vee\big) < k\big\}.\]
Then these $s$, $t$ matrices are invariant (up to sign) under the action of the affine Weyl group
\[W^a = W \ltimes kQ^\vee.\]
In fact, $C$ is simply the fundamental domain $P/W^a$. Using this fact, we can manipulate the form of $Z_{G_k}(Y)$ to write it in a Gauss sum reciprocity-friendly way:
\begin{gather}
\text{\eqref{eq2.1}} = \frac{1}{|W^V|} \sum_{\text{coloring}\in W(C)^{V}}\prod_{v\in V}\mathcal{V}_v\prod_{e\in E}\mathcal{E}_e \nonumber\\
\hphantom{\text{\eqref{eq2.1}}}{} = \frac{1}{|W^V|} q^{-\frac{\sum a_j}{2}(\rho,\rho)}\left( \frac{{\rm i}^{|\Delta_+|}}{|P/kQ^\vee|^{1/2}}\right)^{|V| + 1} \sum_{\lambda\in W(C)^{V}}\prod_{v\in V} \left(\sum_{w\in W}(-1)^{l(w)}q^{(\lambda_v,w(\rho))}\right)^{2-\deg v}\nonumber\\
\hphantom{\text{\eqref{eq2.1}}=}{} \times q^{\frac{a_v}{2}(\lambda_v,\lambda_v)} \prod_{(u_1,u_2)\in E}\sum_{w\in W}(-1)^{l(w)}q^{(w(\lambda_{u_1}),\lambda_{u_2})} \nonumber\\
\hphantom{\text{\eqref{eq2.1}}}{} = \frac{1}{|W|} q^{-\frac{\sum a_j}{2}(\rho,\rho)}\left( \frac{{\rm i}^{|\Delta_+|}}{|P/kQ^\vee|^{1/2}}\right)^{|V| + 1}\nonumber\\
\hphantom{\text{\eqref{eq2.1}}=}{} \times \sum_{\lambda\in W(C)^{V}}\prod_{v\in V} \left(\sum_{w\in W}(-1)^{l(w)}q^{(\lambda_v,w(\rho))}\right)^{2-\deg v} q^{\frac{1}{2}(\lambda, {\rm B}_s\lambda)}. \label{midstep}
\end{gather}

To extend the range of summation $W(C)$ to $P/kQ^\vee$, i.e., to make sense of the summation even for colors upon which the action of $W^a$ is not free, we need to regularize the linear term $\prod\limits_{v\in V} \big(\sum\limits_{w\in W}(-1)^{l(w)}q^{(\lambda_v,w(\rho))}\big)^{2-\deg v}$.
Let $\omega_1, \dots, \omega_r\in P$ be a $\mathbb{Z}$-linear basis of $P$ (e.g., fundamental weights). We can then write $\lambda_v = \sum\limits_{j = 1}^{r}n_{vj}\omega_j$ for some $n_{v1}, \dots, n_{vr}\in \mathbb{Z}$. Thanks to Weyl denominator formula, we have
\begin{gather}
 \prod_{v\in V} \left(\sum_{w\in W}(-1)^{l(w)}q^{(\lambda_v,w(\rho))}\right)^{2-\deg v} \nonumber\\
\qquad{} = \prod_{v\in V} \left( \prod_{\alpha\in \Delta_+}\big(q^{\frac{(\lambda_v,\alpha)}{2}} - q^{-\frac{(\lambda_v,\alpha)}{2}}\big) \right)^{2-\deg v} \nonumber\\
\qquad{} = \prod_{v\in V} \left( \prod_{\alpha\in \Delta_+}\left(\prod_{1\leq j\leq r}x_{vj}^{\frac{(\omega_j,\alpha)}{2}} - \prod_{1\leq j\leq r}x_{vj}^{-\frac{(\omega_j,\alpha)}{2}}\right) \right)^{2-\deg v}\bigg|_{x_{vj} = q^{n_{vj}}}.\label{regularize}
\end{gather}
In case $\deg v > 2$, this expression can be singular only when
\[\left| \prod_{1\leq j\leq r}y_j^{\frac{(\omega_j,\alpha)}{2}}\right| = 1\]
for some $\alpha\in \Delta_+$.
In terms of new variables $z_j := \log |x_j|$, this is simply
\[\sum_{1\leq j\leq r} (\omega_j,\alpha) z_j = 0.\]
These are precisely the walls (hyperplanes) for Weyl reflections. Deforming the origin $z_1 = \cdots = z_r = 0$ to a complement of these walls is the same as a choice of a~Weyl chamber. Moreover, for each choice of such a Weyl chamber, we can expand~(\ref{regularize}) as a geometric series.
Therefore, we can regularize the linear term by taking an average of $|W|$ number of $q$-series, each determined by a choice of a Weyl chamber. To sum up, we can re-express the linear term in the following form:
\begin{gather}\label{regularization}
\prod_{v\in V} \left(\sum_{w\in W}(-1)^{l(w)}q^{(\lambda_v,w(\rho))}\right)^{2-\deg v}
\overset{\text{regularize}}{=\joinrel=\joinrel=\joinrel\Rightarrow} \frac{1}{|W^V|} \sum_{\ell\in \delta + Q^{V}} n_\ell q^{(\lambda,\ell)}
\end{gather}
with $\delta_v = (2-\deg v)\rho \mod Q$ and $n_\ell\in \mathbb{Z}$.

With this regularization in hand, we extend the range of summation in \eqref{midstep} and apply the Gauss sum reciprocity.\footnote{See \cite{DT} for the version of Gauss sum reciprocity formula we use.} Then, if we choose $n\in \mathbb{Z}^+$ to be such that $nP\subseteq Q^\vee\subseteq P$, we get
\begin{gather*}
 \frac{1}{|W|} q^{-\frac{\sum a_j}{2}(\rho,\rho)}\left( \frac{{\rm i}^{|\Delta_+|}}{|P/kQ^\vee|^{1/2}}\right)^{|V| + 1} \frac{1}{|(Q^\vee/nP)^V|}\nonumber\\
\qquad\quad{}\times \sum_{\lambda \in P^{V}/nkP^{V}}q^{\frac{1}{2}(\lambda,{\rm B}\lambda)} \cdot \left(\frac{1}{|W^V|} \sum_{\ell\in \delta + Q^{V}}n_\ell q^{(\lambda,\ell)}\right) \nonumber\\
\qquad{}= \frac{(-1)^{|\Delta_+||V|}}{|W|^{|V|+1}}\left(\frac{{\rm i}^{|\Delta_+|}}{|P/Q^\vee|^{1/2}} \right)^{-|V|+1} q^{-\frac{\sum a_j}{2}(\rho,\rho)}k^{-\frac{r}{2}} \frac{{\rm e}^{\frac{\pi{\rm i}}{4}\sigma({\rm B})}}{|\det {\rm B}|^{1/2}} \nonumber\\
\qquad\quad{}\times \sum_{a\in (P^\vee)^{V}/{\rm B}(P^\vee)^{V}}{\rm e}^{-\pi{\rm i} k(a,{\rm B}^{-1}a)}\sum_{b\in (Q^{V}+\delta)/{\rm B}Q^{V}}{\rm e}^{-2\pi{\rm i}(a,{\rm B}^{-1}b)} \sum_{\ell\in {\rm B}Q^{V}+b}n_\ell q^{-\frac{1}{2}(\ell,{\rm B}^{-1}\ell)}.
\end{gather*}
Note that $a$ and $b$ takes values in different sets: $a\in \big(P^\vee\big)^{V}/{\rm B}\big(P^\vee\big)^{V}$ while $b\in \big(Q^{V}+\delta\big)/{\rm B}Q^{V}$. The $a$ labels should be understood as `Abelian flat connections' and the $b$ labels are `$\mathcal{B}^G$-structures' of Section~\ref{spint}.

To summarize, we have heuristically decomposed the WRT invariant $Z_{G_k}(Y)$ into some linear combinations of $q$-series. When compared with Conjecture~2.1 of~\cite{GPPV}, this suggests the following expression of $\hat{Z}^G$ for negative definite plumbed manifolds:\footnote{Our decomposition was only a heuristic. Although we believe that this heuristic can be made rigorously to prove Conjecture~2.1 of~\cite{GPPV} for negative definite plumbings on trees, we do not bother to do so, because Conjecture~2.1 of~\cite{GPPV} should be modified anyway for general $3$-manifolds, see~\cite{CCFGH, CGPS, GM}.}
\begin{gather*}
\hat{Z}_b^G(Y;q) \cong |W|^{-|V|}q^{-\frac{\operatorname{Tr}{\rm B}}{2}(\rho,\rho)}\sum_{\ell\in {\rm B}Q^{V} + b}n_\ell q^{-\frac{1}{2}(\ell,{\rm B}^{-1}\ell)}\;\in |W|^{-|V|} q^{\Delta_b}\mathbb{Z}[[q]], 
\end{gather*}
where
\begin{gather*}
b\in \big(Q^{V}+\delta\big)/{\rm B}Q^{V},
\qquad
\Delta_b = -\frac{\operatorname{Tr}{\rm B}}{2}(\rho,\rho) + \min_{\ell\in {\rm B}Q^{V}+b}-\frac{1}{2}\big(\ell,{\rm B}^{-1}\ell\big)\in \mathbb{Q},
\end{gather*}
and the integers $n_\ell$ are determined as in (\ref{regularization}). Moreover the higher rank analog of the $S_{ab}$ matrix in Conjecture~2.1 of~\cite{GPPV} is $S_{ab} = {\rm e}^{-2\pi{\rm i}(a,{\rm B}^{-1}b)}$ before folding by Weyl symmetry.
Note that because we have ignored the framing factor in (\ref{eq2.1}), $\Delta_b$ as defined above is only meaningful up to overall shift; only the differences $\Delta_b - \Delta_{b'}$ are meaningful values.
It is easy to check that the expression we have just arrived is equivalent to Definition \ref{def:Zhatmaindef}.

\section{Comparison with Chung's paper}
In \cite{C}, Chung studied $\hat{Z}$ for some Seifert manifolds, with $G={\rm SU}(N)$. His approach was to start from Marino's integral expression for $Z_a = \sum_{b}S_{ab}\hat{Z}_b$ in \cite{M} and decompose it into $\hat{Z}_b$'s by collecting terms whose $q$-degrees differ by an integer.
In this section we compare our result with Chung's result in some examples.
\begin{itemize}\itemsep=0pt
\item $Y = M\big({-}1;\frac{1}{2},\frac{1}{3},\frac{1}{7}\big)$. In this case $|H_1(Y)| = 1$ and there's only one homological block.
\begin{table}[h!]\centering
\begin{tabular}{c | c }
$G$ & $\hat{Z}_0(Y) \cong$\\
\hline
${\rm SU}(2)$ & $1 -q -q^5 +q^{10} -q^{11} +q^{18} +q^{30} -q^{41} +q^{43} -q^{56} -q^{76} +q^{93} -\cdots$\tsep{2pt}\bsep{2pt}\\
${\rm SU}(3)$ & $1 -2q +2q^3 +q^4 -2q^5 -2q^8 + 4q^9 + 2q^{10} - 4q^{11} +2q^{13} -6q^{14} +2q^{15} -\cdots$\tsep{2pt}\bsep{2pt}\\
${\rm SU}(4)$ & $1 -3q +q^2 +3q^3 -3q^5 -q^6 -q^7 -5q^8 +15q^9 +5q^{10} -11q^{11} -q^{12} +\cdots$
\end{tabular}
\end{table}
\item $Y = M\big({-}1;\frac{1}{2},\frac{1}{5},\frac{2}{7}\big)$. Again, $|H_1(Y)| = 1$ and there is only one homological block.
\begin{table}[h!]\centering
\begin{tabular}{c | c }
$G$ & $\hat{Z}_0(Y) \cong$\\
\hline
${\rm SU}(2)$ & $1 -q^{3} -q^{5} +q^{12} -q^{23} +q^{36} +q^{42} -q^{59} +q^{81} -q^{104} -q^{114} +q^{141} -\cdots$\tsep{2pt}\bsep{2pt}\\
${\rm SU}(3)$ & $1 -2q^3 -2q^5 +2q^6 +2q^9 +q^{12} -2q^{14} +2q^{15} -2q^{18} -3q^{20} +6q^{21} -4q^{23} -\cdots$\tsep{2pt}\bsep{2pt}\\
${\rm SU}(4)$ & $1 -3q^3 -3q^5 +5q^6 -q^7 +2q^8 +3q^9 -q^{10} -q^{12} +2q^{13} -6q^{14} +2q^{15} -\cdots$
\end{tabular}
\end{table}
\item $Y = M\big({-}1;\frac{1}{3},\frac{1}{5},\frac{3}{7}\big)$. In this case $|H_1(Y)| = 4$.
\begin{table}[h!]\centering
\begin{tabular}{c | c }
$G$ & $\hat{Z}_b(Y) \cong$\\
\hline
${\rm SU}(2)$ & $1 +q^4 +q^{16} -q^{68} +q^{144} -q^{260} -q^{320} -q^{356} +q^{484} +q^{528} +q^{612} -q^{832} +\cdots$\tsep{2pt}\\
& $-q^{15/4}\big(1 +q^6 +q^{10} +q^{12} -q^{44} -q^{48} -q^{58} -q^{88} +q^{122} +q^{164} +q^{182} +\cdots\big)$\\
& $q^{13/2}\big(1 -q^{32} -q^{56} -q^{72} +q^{136} +q^{160} +q^{208} -q^{344} +q^{496} -q^{696} -q^{792} -\cdots\big)$\bsep{2pt}\\
\hline
${\rm SU}(3)$ & $1 +3q^4 +2q^{12} +3q^{16} +2q^{28} +2q^{48} +2q^{52} +q^{64} +4q^{68} +4q^{80} +4q^{92} - \cdots$\tsep{2pt}\\
 & $-q^{-7/4}\big(2 +2q +2q^3 -4q^5 +2q^6 -2q^7 +4q^9 +2q^{10} +4q^{12} +2q^{13} -2q^{14} +\cdots\big)$\\
 & $-q^{-7}\big(1 +2q^6 -2q^8 +2q^{12} -2q^{14} +2q^{18} -2q^{20} +q^{24} -2q^{26} +2q^{28} +\cdots\big)$\\
 & $q^{-21/4}\big(1 +q -q^2 +q^4 -q^5 +q^7 +q^{10} +q^{13} -2q^{14} -q^{15} +2q^{16} +\cdots\big) \times 2$\bsep{2pt}\\
\hline
${\rm SU}(4)$ & $1 +q^{12} +8q^{16} +3q^{20} +16q^{24} +11q^{28} +15q^{32} +4q^{36} +26q^{40} +5q^{44} +\cdots$\tsep{2pt}\\
& $-q^{-15/4}\big(1 +2q -2q^2 +2q^4 -2q^5 +q^6 +4q^7 -2q^8 +3q^{10} +6q^{13} -\cdots\big)$\\
& $q^{-1}\big(2 -2q^2 -2q^4 -3q^6 -2q^8 -5q^{10} -14q^{12} -5q^{14} -4q^{16} -12q^{18} +\cdots\big)$\\
& $q^{-1/4}\big(2 +2q -2q^2 +2q^3 +q^4 -2q^5 +6q^6 +2q^7 +8q^8 +6q^9 -\cdots\big) \times 2$\\
& $-q^{3/2}\big(1 +3q^4 +2q^8 +6q^{12} +6q^{16} +4q^{20} +9q^{24} +9q^{28} +11q^{32} +\cdots\big) \times 2$\\
& $q^{-2}\big(1 +2q^4 +3q^8 +6q^{16} +6q^{20} +11q^{24} +17q^{32} +10q^{36} +9q^{40} +14q^{48} +\cdots\big)$\\
& $-q^{-15/4}\big(2 +q -2q^2 +q^3 +2q^4 -q^5 +4q^6 +4q^7 +2q^8 -2q^9 +\cdots\big)$\\
& $q^{-11/2}\big(1 +3q^6 -4q^8 +7q^{12} -2q^{14} -q^{16} +6q^{18} +3q^{20} +7q^{22} +4q^{24} +\cdots\big)$
\end{tabular}
\end{table}
\end{itemize}
Observe that \cite{C} agrees with our example computations except in the ${\rm SU}(4)$ case of the last example. Our $\hat{Z}$'s are more refined in a sense that Chung's $\hat{Z}_{1}$ is the sum of our 2nd and 7th ones and Chung's $\hat{Z}_{3}$ is the sum of our 1st and 6th ones. This example illustrates that in general we cannot fully decompose $Z_a$ into $\hat{Z}_b$'s by just collecting terms whose $q$-powers differ by an integer as done in~\cite{C}.\footnote{Still, it is possible to derive our formula (\ref{integralZhat}) from Mari\~{n}o's Chern--Simons matrix model or vice versa in case of Seifert manifolds. This is because Gaussian measure is the same as Laplace transform accompanied by the $S_{ab}$ matrix.}

\subsection*{Acknowledgements}
I would like to thank my advisor Sergei Gukov for his invaluable guidance, as well as Francesca Ferrari, Sarah Harrison, Ciprian Manolescu and Nikita Sopenko for helpful conversations.
Special thanks go to Nikita Sopenko for his kind help with Mathematica coding. I would also like to thank the anonymous referees for useful comments that helped to improve the paper.

The author was supported by Kwanjeong Educational Foundation.

\pdfbookmark[1]{References}{ref}
\LastPageEnding

\end{document}